\documentclass{birkjour}
 \newtheorem{theorem}{Theorem}[section]
 \newtheorem{corollary}[theorem]{Corollary}
 \newtheorem{lemma}[theorem]{Lemma}
 \newtheorem{proposition}[theorem]{Proposition}
 \theoremstyle{definition}
 \newtheorem{definition}[theorem]{Definition}
 \newtheorem{remark}[theorem]{Remark}
 \newtheorem{example}[theorem]{Example}
 \numberwithin{equation}{section}

\begin{document}

%
%
%
%
%
%
%
%
%

\title[Bochner $pg$-frames]
 {Bochner $pg$-frames}

\author [M. Rahmani]{Morteza Rahmani}
\address{Department of Pure Mathematics, \\
Faculty of Mathematical Sciences, \\
University of Tabriz, Tabriz, Iran \\}
\email{m\b{ }rahmani@tabrizu.ac.ir}

\author[M. H. Faroughi]{Mohammad H. Faroughi}

\address{%
Department of Pure Mathematics, \\
Faculty of Mathematical Sciences, \\
University of Tabriz, Tabriz, Iran \\
 and \\ Islamic Azad University-Shabestar Branch,\\ Shabestar, Iran\\
}
\email{mhfaroughi@yahoo.com}

\subjclass{42C15, 46G10}

\keywords{Banach space, Hilbert space, Frame, Bochner measurable, Bochner $pg$-frame, Bochner $pg$-Bessel family, Bochner $qg$-Riesz basis}


\begin{abstract}
In this paper we introduce the concept of Bochner $pg$-frames for Banach spaces. We characterize the Bochner $pg$-frames and specify the optimal bounds of a Bochner $pg$-frame. Then we define a Bochner $qg$-Riesz basis and verify the relations between Bochner $pg$-frames and Bochner $qg$-Riesz bases. Finally, we discuss the perturbation  of Bochner $pg$-frames.
\end{abstract}

\maketitle
\section{Introduction and Preliminaries}
The concept of frames (discrete frames) in Hilbert spaces has been introduced by Duffin and Schaeffer \cite{Duff} in 1952 to study some deep problems in nonharmonic Fourier series. After the fundamental paper \cite{Daub} by Daubechies, Grossmann and Meyer, frame theory began to be widely used, particularly in the more specialized context of wavelet frames and Gabor frames. Frames play a fundamental role in signal processing, image and data compression and sampling theory. They provided an alternative to orthonormal bases, and have the advantage of possessing a certain degree of redundancy. A discrete frame is a countable family of elements in a separable Hilbert space which allows for a stable, not necessarily unique, decomposition of an arbitrary element into an expansion of the frame elements. For more details about discrete frames see \cite{gj91}. Resent results show that frames can provide a universal language in which many fundamental problems in pure mathematics can be formulated: The Kadison-Singer problem in operator algebras, the Bourgain-Tzafriri conjecture in Banach space theory, paving Toeplitz operators in harmonic analysis and many others. Various types of frames have been proposed, for example, $pg$-frames in Banach spaces \cite{pg-frame}, fusion frames \cite{cay}, continuous frames in Hilbert space \cite{hy43}, continuous frames in Hilbert spaces \cite{hir103}, continuous $g$-frames in Hilbert spaces \cite{cg-frame}, $(p,Y)$-operator frames for a Banach space \cite{Hu1}.

This paper is organized as follows. In Section 2, we introduce the concept of Bochner $pg$-frames for Banach spaces. Actually, continuous frames motivate us to introduce this kind of frames and analogous to continuous frames which are generalized version of discrete frames, we want to generalize $pg$-frames in a continuous sense. Like continuous frames, these frames can be used in those areas that we need generalized  frames in a continuous aspect.   Also, we  define corresponding operators (synthesis, analysis and frame operators) and discuss their characteristics and properties. In Section 3, we define a Bochner $qg$-Riesz basis and verify its relations by Bochner $pg$-frames. Finally, Section 4 is devoted to perturbation of Bochner $pg$-frames.

Throughout this paper, $X$ and $H$ will be a Banach space and a Hilbert space, respectively, and $\{H_{\omega}\}_{\omega\in \Omega}$ is a family of Hilbert spaces.

Suppose that $(\Omega,\Sigma,\mu)$ is a measure space, where $\mu$ is a positive measure.

The following definition introduces  Bochner measurable functions.

\begin{definition}
A function $f:\Omega\longrightarrow X$ is called  Bochner measurable if there exists a sequence of simple functions $\{f_{n}\}_{n=1}^{\infty}$ such that
$$\lim_{n\rightarrow\infty}\|f_{n}(\omega)-f(\omega)\|=0, ~~a.e.~[\mu].$$
\end{definition}

\begin{definition}
If $\mu$ is a measure on $(\Omega,\Sigma)$ then $X$ has the Radon-Nikodym property with respect to $\mu$ if for every countably additive vector measure $\gamma$ on $(\Omega,\Sigma)$ with values in $X$ which has bounded variation and is absolutely continuous with respect to $\mu$, there is a Bochner integrable function $g:\Omega\longrightarrow X$ such that
$$\gamma(E)=\int_{E}g(\omega)d\mu(\omega)$$
for every  set $E\in\Sigma$.

 A Banach space $X$ has the Radon-Nikodym property if $X$ has the Radon-Nikodym property with respect to every finite measure. Spaces with Radon-Nikodym property include separable dual spaces and reflexive spaces, which include, in particular, Hilbert spaces.
\end{definition}

\begin{remark}\label{00}
Suppose that $(\Omega,\Sigma,\mu)$ is a measure
space and $X^{\ast}$ has the Radon-Nikodym property. Let $1\leq p\leq\infty$. The Bochner space of $L^{p}(\mu,X)$ is defined to be the Banach space of (equivalence classes of) $X$-valued Bochner measurable functions $F$ from $\Omega$ to $X$ for which the norms
$$\|F\|_{p}=(\int_{\Omega}\|F(\omega)\|^{p}d\mu(\omega))^{\frac{1}{p}},\quad 1\leq p<\infty$$
$$\|F\|_{\infty}=ess~sup_{\omega\in \Omega}\|F(\omega)\|,\quad p=\infty$$
are finite.
In \cite{die1}, \cite{c1} and \cite[p.51]{Fleming} it is proved that if $1\leq p<\infty$ and $q$ is such that $\frac{1}{p}+\frac{1}{q}=1$, then $L^{q}(\mu,X^{\ast})$ is isometrically
isomorphic to $(L^{p}(\mu,X))^{\ast}$ if and only if $X^{\ast}$ has
the Radon-Nikodym property. This isometric
isomorphism is the mapping
$$\psi:L^{q}(\mu,X^{\ast})\rightarrow (L^{p}(\mu,X))^{\ast}$$$$g\mapsto\psi(g),$$ where the mapping $\psi(g)$ is defined on $L^{p}(\mu,X)$ by $$\psi(g)(f)=\int_{\Omega}g(\omega)(f(\omega))d\mu(\omega),\quad f\in L^{p}(\mu,X).$$
So  for all $f\in L^{p}(\mu,X)$ and $g\in L^{q}(\mu,X^{\ast})$ we have
$$<f,\psi(g)>=\int_{\Omega}<f(\omega),g(\omega)>d\mu(\omega).$$
In the following, we use the notation $<f,g>$ instead of $<f,\psi(g)>$, so for all $f\in L^{p}(\mu,X)$ and $g\in L^{q}(\mu,X^{\ast})$ $$<f,g>=\int_{\Omega}<f(\omega),g(\omega)>d\mu(\omega).$$
Particularly, if $H$ is a Hilbert space then $(L^{p}(\mu,H))^{*}$ is isometrically isomorphic to $L^{q}(\mu,H)$. So, for all $f\in L^{p}(\mu,H)$ and $g\in L^{q}(\mu,H)$
$$<f,g>=\int_{\Omega}<f(\omega),g(\omega)>d\mu(\omega),$$
in which $<f(\omega),g(\omega)>$ dose not mean the inner product of elements $f(\omega)$, $g(\omega)$ in $H$, but
$$<f(\omega),g(\omega)>=\nu(g(\omega))(f(\omega)),$$
where $\nu:H\longrightarrow H^{*}$ is the isometric isomorphism between $H$ and $H^{*}$, for more details refer to \cite[p.54]{Murphy}.
\end{remark}

We will use the following lemma which is proved in \cite{Heuser}.
\begin{lemma}\label{0}
If $U:X\longrightarrow Y$ is a bounded operator from a Banach space $X$ into a Banach space $Y$ then its adjoint $U^{*}:Y^{*}\longrightarrow X^{*}$
is surjective if and only if $U$ has a bounded inverse on $R_{U}$.
\end{lemma}

Note that for a collection $\{H_{\beta}\}_{\beta\in \mathfrak{B}}$ of Hilbert spaces, we can suppose that there exists a Hilbert space $K$ such that for all $\beta\in \mathfrak{B}$, $H_{\beta}\subseteq K$, where
$K=\oplus_{\beta \in \mathfrak{B}} H_{\beta}$ is the direct sum of $\{H_{\beta}\}_{\beta\in \mathfrak{B}}$, see 3.1.5 in \cite[p.81]{hir08}.

\section{Bochner $pg$-frames}
Bochner spaces are often used in the functional analysis approach to the study of partial differential equations that depend on time, e.g. the heat equation: if the temperature $g(t,x)$ is a scalar function of time and space, one can write $(f(t))(x):=g(t,x)$ to make $f$ a function of time, with $f(t)$ being a function of space, possibly in some Bochner space. Now, we intend to use this space to define a new kind of frames which contain all of continuous and discrete frames, in other words we will generalize the $g$-frames to a continuous case that is constructed on concept of Bochner spaces. Of course, this new frame can be useful in function spaces and operator theory to gain some general results that are achieved by $g$-frames or discrete frames.

\subsection{Bochner $pg$-frames and corresponding operators}
We start with the definition of Bochner $pg$-frames. Then we will give some characterizations of these frames.
\begin{definition}
Let $1<p<\infty$. The family $\{\Lambda_{\omega}\in B(X,H_{\omega}): \omega \in \Omega\}$ is a Bochner $pg$-frame for $X$ with respect to $\{H_{\omega}\}_{\omega\in \Omega}$ if:
\\$(i)$ For each $x\in X$, $\omega \longmapsto \Lambda_{\omega}(x)$ is Bochner measurable,
\\$(ii)$ there exist positive constants $A$ and $B$ such that
\begin{eqnarray}\label{df1}
A\|x\|\leq (\int_{\Omega}\|\Lambda_{\omega}(x)\|^{p}d\mu(\omega))^{\frac{1}{p}}\leq B\|x\|,\quad x\in X.
\end{eqnarray}
\end{definition}
$A$ and $B$ are called the lower and upper Bochner $pg$-frame bound, respectively.
We call that $\{\Lambda_{\omega}\}_{\omega \in \Omega}$ is a tight Bochner $pg$-frame if $A$ and $B$ can be chosen such that $A=B$  and a Parseval  Bochner $pg$-frame if $A$ and $B$ can be chosen such that $A=B=1$. If for each $\omega \in \Omega$, $H_{\omega}=H$ then $\{\Lambda_{\omega}\}_{\omega \in \Omega}$ is called a Bochner $pg$-frame for $X$ with respect to $H$. A family $\{\Lambda_{\omega}\in B(X,H_{\omega}): \omega \in \Omega\}$ is called a Bochner $pg$-Bessel family for $X$ with respect to $\{H_{\omega}\}_{\omega \in \Omega}$ if the right inequality in (\ref{df1}) holds. In this case, $B$ is called the Bessel bound.

\begin{example}
Let $\{f_i\}_{i \in I}$ be a frame for Hilbert space $H$, $\Omega=I$ and $\mu$ be a counting measure on $\Omega$. Set
$$\Lambda_i:H\longrightarrow \mathbb{C}$$
$$\Lambda_i(h)=<h,f_i>,~h\in H.$$
Then $\{\Lambda_{i}\}_{i \in I}$ is a Bochner $pg$-frame for $H$ with respect to $\mathbb{C}$.
\end{example}

\begin{example}
Let $\Omega=\{a,b,c\}$, $\Sigma=\{\emptyset, \{a,b\},\{c\},\Omega\}$ and $\mu:\Sigma\longrightarrow [0,\infty]$ be a measure such that $\mu(\emptyset)=0$, $\mu(\{a,b\})=1$,  $\mu(\{c\})=1$ and $\mu(\Omega)=2$. Assume that $X=L^p(\Omega)$ and $\{H_{\omega}\}_{\omega\in \Omega}$ is a family of arbitrary Hilbert spaces and consider a fixed family $\{h_\omega\}_{\omega\in \Omega}\subseteq \{H_\omega\}_{\omega\in \Omega}$  such that $\|h_\omega\|=1$, $\omega\in \Omega$. Suppose that
$$\Lambda_\omega:L^p(\Omega)\longrightarrow H_\omega,$$
$$\Lambda_\omega(\varphi)=\varphi(c)h_\omega.$$

It is clear that $\Lambda_\omega$'s are bounded and for each $\varphi\in L^p(\Omega)$, $\omega \longmapsto \Lambda_{\omega}(\varphi)$ is Bochner measurable.
Also,
\begin{align*}
\int_{\Omega}\|\Lambda_{\omega}(x)\|^{p}d\mu(\omega)=\int_{\Omega}|\varphi(c)|^pd\mu(\omega)=|\varphi(c)|^p\mu(\Omega)=2|\varphi(c)|^p.
\end{align*}
So, $\{\Lambda_{\omega}\}_{\omega \in \Omega}$ is a Bochner $pg$-frame for $L^p(\Omega)$ with respect to $\{H_{\omega}\}_{\omega\in \Omega}$.
\end{example}

Now, we state the definition of some common corresponding operators for a Bochner $pg$-frame.

\begin{definition}
Let $\{\Lambda_{\omega}\}_{\omega \in \Omega}$ be a Bochner $pg$-Bessel family for $X$ with respect to $\{H_{\omega}\}_{\omega \in \Omega}$ and $q$ be the conjugate exponent of $p$. We define the operators $T$ and $U$, by
\begin{eqnarray}\label{T}
T:L^{q}(\mu,\oplus_{\omega \in \Omega} H_{\omega})\longrightarrow X^{*}
\end{eqnarray}
\begin{align*}
<x,TG>=\int_{\Omega}<\Lambda_{\omega}(x),G(\omega)>d\mu(\omega),\quad  x\in X,~G\in L^{q}(\mu,\oplus_{\omega \in \Omega} H_{\omega}),
\end{align*}
\begin{eqnarray}\label{U}
U:X\longrightarrow L^{p}(\mu,\oplus_{\omega \in \Omega} H_{\omega})
\end{eqnarray}
\begin{align*}
<Ux,G>=\int_{\Omega}<\Lambda_{\omega}(x),G(\omega)>d\mu(\omega),\quad  x\in X,~G\in L^{q}(\mu,\oplus_{\omega \in \Omega} H_{\omega}).
\end{align*}
The operators $T$ and $U$ are called the synthesis and analysis operators of $\{\Lambda_{\omega}\}_{\omega \in \Omega}$, respectively.
\end{definition}

The following proposition shows these operators are bounded. It is analogous to Theorem 3.2.3 in \cite{gj91}.
\begin{proposition}\label{3}
Let $\{\Lambda_{\omega}\in B(X,H_{\omega}): \omega \in \Omega\}$ be a  Bochner $pg$-Bessel family for $X$ with respect to $\{H_{\omega}\}_{\omega \in \Omega}$ and with Bessel bound $B$. Then the operators $T$ and $U$ defined by (\ref{T}) and (\ref{U}), respectively, are well-defined and bounded with $\|T\|\leq B$ and $\|U\|\leq B$.
\end{proposition}
\begin{proof}
Suppose that $\{\Lambda_{\omega}\}_{\omega \in \Omega}$ is a Bochner $pg$-Bessel family with bound $B$ and $q$ is the conjugate exponent of $p$. We show that for all $x\in X$ and all $G\in L^{q}(\mu,\oplus_{\omega \in \Omega} H_{\omega})$, the mapping $\omega\longmapsto <\Lambda_{\omega}(x),G(\omega)>$ is measurable.  For all $x\in X$ and $G\in L^{q}(\mu,\oplus_{\omega \in \Omega} H_{\omega})$, $\omega \longmapsto \Lambda_{\omega}(x)$ and $G$ are Bochner measurable, so there are  sequences of simple functions $\{\lambda_{n}\}_{n=1}^{\infty}$ and $\{g_{n}\}_{n=1}^{\infty}$ such that
$$\lim_{n\rightarrow\infty}\|\lambda_{n}(\omega)-\Lambda_{\omega}(x)\|=0, \quad a.e.~[\mu],$$
$$\lim_{n\rightarrow\infty}\|g_{n}(\omega)-G(\omega)\|=0, \quad a.e.~[\mu].$$
For each $n$, $<\lambda_{n},g_{n}>$ is a  simple function and
\begin{align*}
&|<\Lambda_{\omega}(x),G(\omega)>-<\lambda_{n}(\omega),g_{n}(\omega)>|
\\\leq&|<\Lambda_{\omega}(x)-\lambda_{n}(\omega),G(\omega)>|+|<\lambda_{n}(\omega),g_{n}(\omega)-G(\omega)>|
\\\leq&\|\Lambda_{\omega}(x)-\lambda_{n}(\omega)\|\|G(\omega)\|+\|\lambda_{n}(\omega)\|\|g_{n}(\omega)-G(\omega)\|.
\end{align*}
So $$\lim_{n\rightarrow\infty}|<\Lambda_{\omega}(x),G(\omega)>-<\lambda_{n}(\omega),g_{n}(\omega)>|=0$$ and $\omega\longmapsto <\Lambda_{\omega}(x),G(\omega)>$ is measurable.
\\For each $x\in X$ and $G\in L^{q}(\mu,\oplus_{\omega \in \Omega} H_{\omega})$, we have
\begin{align*}
|<x,TG>|=&|\int_{\Omega}<\Lambda_{\omega}(x),G(\omega)>d\mu(\omega)|
\\\leq& \int_{\Omega}\|\Lambda_{\omega}x\|\|G(\omega)\|d\mu(\omega)
\\\leq& (\int_{\Omega}\|\Lambda_{\omega}x\|^{p}d\mu(\omega))^{\frac{1}{p}}
(\int_{\Omega}\|G(\omega)\|^{q}d\mu(\omega))^{\frac{1}{q}}
\\\leq&B\|x\|\|G\|_{q}.
\end{align*}
Thus $T$ is well-defined and $\|T\|\leq B$. By a similar discussion, $U$ is well-defined and $\|U\|\leq B$.
\end{proof}

The following proposition provides us with a concrete formula for the analysis operator.
\begin{proposition}\label{T*}
If $\{\Lambda_{\omega}\}_{\omega \in \Omega}$ is a Bochner $pg$-Bessel family for $X$ with respect to $\{H_{\omega}\}_{\omega \in \Omega}$ then for all $x\in X$, $(Ux)(\omega)=\Lambda_{\omega}x,\,~ a.e.~[\mu]$.
\end{proposition}

\begin{proof}
Let $q$ be the conjugate exponent of $p$ and $x\in X$. For all  $G\in L^{q}(\mu,\oplus_{\omega \in \Omega} H_{\omega})$, we have
\begin{eqnarray*}
<Ux,G>&=&{\int_{\Omega}<\Lambda_{\omega}(x),G(\omega)>d\mu(\omega)}
\\&=&<\{\Lambda_{\omega}x\}_{\omega \in \Omega},G>.
\end{eqnarray*}
So $<Ux-\{\Lambda_{\omega}x\}_{\omega \in \Omega},G>=0$, for all $G\in L^{q}(\mu,\oplus_{\omega \in \Omega} H_{\omega})$.
There exists $G\in L^{q}(\mu,\oplus_{\omega \in \Omega} H_{\omega})$ such that $\|G\|_{q}=1$ and $$<Ux-\{\Lambda_{\omega}x\}_{\omega \in \Omega},G>=\|Ux-\{\Lambda_{\omega}x\}_{\omega \in \Omega}\|_{p},$$
which implies $\|Ux-\{\Lambda_{\omega}x\}_{\omega \in \Omega}\|_{p}=0$. Therefore $(Ux)(\omega)=\Lambda_{\omega}x,\,~ a.e.~[\mu]$.
\end{proof}

The following proposition shows that it is enough to check the Bochner $pg$-frame conditions on a dense subset. The discrete version of this proposition is available in \cite[Lemma 5.1.7]{gj91}.

\begin{proposition}
Suppose that $(\Omega,\Sigma,\mu)$ is a measure space where $\mu$ is $\sigma$-finite. Let $\{\Lambda_{\omega}\in B(X,H_{\omega}): \omega \in \Omega\}$ be a family such that for each $x\in X$, $\omega \longmapsto \Lambda_{\omega}(x)$ is Bochner measurable and assume that there exist positive constants $A$ and $B$ such that (\ref{df1}) holds for all $x$ in a dense subset $V$ of $X$. Then $\{\Lambda_{\omega}\}_{\omega \in \Omega}$ is a Bochner $pg$-frame for $X$ with respect to $\{H_{\omega}\}_{\omega\in \Omega}$  with  bounds $A$ and $B$.
\end{proposition}

\begin{proof}
Let $\{\Omega_{n}\}_{n=1}^\infty$ be a family of disjoint measurable subsets of $\Omega$ such that $\Omega=\bigcup_{n=1}^{\infty}\Omega_{n}$ with $\mu(\Omega_{n})<\infty$ for each $n\geq1$. Let $x\in X$ and assume without loss of generality $\|\Lambda_{\omega}x\|\neq 0,\,\omega\in\Omega$. Let
$$\Delta_{m}^{x}=\{\omega \in \Omega~|~m-1<\|\Lambda_{\omega}x\|\leq m\},\quad m=0,1,2,...~. $$
It is clear that for each $m=0,1,2,...$~, $\Delta_{m}^{x}\subseteq\Omega$ is measurable and $\Omega=\bigcup_{m=0,n=1}^{\infty}(\Delta_{m}^{x}\cap\Omega_{n})$, where $\{\Delta_{m}^{x}\cap\Omega_{n}\}_{n=1,m=1}^{\infty\quad\infty}$ is a family of disjoint and measurable subsets of $\Omega$. If $\{\Lambda_{\omega}\}_{\omega \in \Omega}$ is not a Bochner $pg$-Bessel family for $X$ then there exists $x\in X$ such that
$$(\int_{\Omega}\|\Lambda_{\omega}(x)\|^{p}d\mu(\omega))^{\frac{1}{p}}> B\|x\|.$$
So
$$\sum_{m,n}\int_{\Delta_{m}^{x}\cap\Omega_{n}}\|\Lambda_{\omega}(x)\|^{p}d\mu(\omega)>B^{p}\|x\|^{p}$$
and there exist finite sets $I$ and $J$ such that
\begin{eqnarray}\label{dense}
\sum_{m\in I}\sum_{n\in J}\int_{\Delta_{m}^{x}\cap\Omega_{n}}\|\Lambda_{\omega}(x)\|^{p}d\mu(\omega)>B^{p}\|x\|^{p}.
\end{eqnarray}
Let $\{x_{k}\}_{k=1}^{\infty}$ be a sequence in $V$ such that $x_{k}\rightarrow x$ as $k\rightarrow\infty$. The assumption implies that
$$\sum_{m\in I}\sum_{n\in J}\int_{\Delta_{m}^{x}\cap\Omega_{n}}\|\Lambda_{\omega}(x_{k})\|^{p}d\mu(\omega)\leq B^{p}\|x_{k}\|^{p},$$
which is a contradiction to (\ref{dense})(by the Lebesgue's Dominated Convergence Theorem). So $\{\Lambda_{\omega}\}_{\omega \in \Omega}$ is a $pg$-Bessel family for $X$ with respect to $\{H_{\omega}\}_{\omega\in \Omega}$ and Bessel bound $B$.
Now, we show that
$$(\int_{\Omega}\|\Lambda_{\omega}(x_{k})\|^{p}d\mu(\omega))^{\frac{1}{p}}\rightarrow (\int_{\Omega}\|\Lambda_{\omega}(x)\|^{p}d\mu(\omega))^{\frac{1}{p}}.$$
Since $\{\Lambda_{\omega}\}_{\omega \in \Omega}$ is a $pg$-Bessel family for $X$, the operator $U$ defined by (\ref{U}) is well defined and bounded.
Assume that $q$ is the conjugate exponent of $p$ and let
$$G:\Omega\longrightarrow \oplus_{\omega \in \Omega} H_{\omega},\quad G(\omega)=\|\Lambda_{\omega}x\|^{\frac{p-q}{q}}\Lambda_{\omega}x,\,\omega \in \Omega,$$
and
$$G_{k}:\Omega\longrightarrow \oplus_{\omega \in \Omega} H_{\omega},\quad G_{k}(\omega)=\|\Lambda_{\omega}x_{k}\|^{\frac{p-q}{q}}\Lambda_{\omega}x_{k},\,\omega \in \Omega.$$
It is obvious that for each $k\in \mathbb{N}$, $G_{k}$ and $G$ belong to $L^{q}(\mu,\oplus_{\omega \in \Omega} H_{\omega})$, and we have
$$<Ux,G>=\int_{\Omega}<\Lambda_{\omega}(x),G(\omega)>d\mu(\omega)=\int_{\Omega}\|\Lambda_{\omega}(x)\|^{p}d\mu(\omega),$$
$$<Ux_{k},G_{k}>=\int_{\Omega}<\Lambda_{\omega}(x_{k}),G_{k}(\omega)>d\mu(\omega)=\int_{\Omega}\|\Lambda_{\omega}(x_{k})\|^{p}d\mu(\omega).$$
Since $\lim_{k\rightarrow \infty}\|(G_{k}-G)(\omega)\|^{q}=0 $ and
$$\|(G_{k}-G)(\omega)\|^{q}\leq (\|G_{k}(\omega)\|+\|G(\omega)\|)^{q}\leq 2^{q-1}(\|G_{k}(\omega)\|^{q}+\|G(\omega)\|^{q}),$$
so by  the Lebesgue's Dominated Convergence Theorem
$$\lim_{k\rightarrow \infty}\int_{\Omega}\|(G_{k}-G)(\omega)\|^{q}d\mu(\omega)=0.$$
Therefore, $\lim_{k\rightarrow \infty}\|G_{k}-G\|_{q}=0$, hence
\begin{align*}
&|\int_{\Omega}\|\Lambda_{\omega}(x_{k})\|^{p}d\mu(\omega)-\int_{\Omega}\|\Lambda_{\omega}(x)\|^{p}d\mu(\omega)|
\\=&|<Ux_{k},G_{k}>-<Ux,G>|
\\\leq&|\int_{\Omega}<\Lambda_{\omega}(x_{k}-x),G_{k}(\omega)>d\mu(\omega)|
\\+&|\int_{\Omega}<\Lambda_{\omega}(x),(G_{k}-G)(\omega)>d\mu(\omega)|\\
\leq&(\int_{\Omega}\|\Lambda_{\omega}(x_{k}-x)\|^{p}d\mu(\omega))^{\frac{1}{p}}(\int_{\Omega}\|G_{k}(\omega)\|^{q}d\mu(\omega))^{\frac{1}{q}}\\
+&(\int_{\Omega}\|\Lambda_{\omega}(x)\|^{p}d\mu(\omega))^{\frac{1}{p}}(\int_{\Omega}\|(G-G_{k})(\omega)\|^{q}d\mu(\omega))^{\frac{1}{q}}\\
\leq& B\|x_{k}-x\|\|G_{k}\|_{q}+B\|x\||\|G_{k}-G\|_{q}.
\end{align*}
By letting $x_{k}\rightarrow x$, the proof is completed.
\end{proof}

\subsection{Characterization of Bochner $pg$-frames}
Now we give some characterizations of Bochner $pg$-frames in terms of their corresponding operators.

At first, we show the next lemma that  is very useful in case of complex valued $L^p$-spaces.

\begin{lemma}\label{4}
Let $(\Omega,\Sigma,\mu)$ be a measure space where $\mu$ is $\sigma$-finite. Let $1< p<\infty$ and $q$ be its conjugate exponent. If $F:\Omega\longrightarrow H$ is Bochner measurable and for each $G\in L^{q}(\mu,H)$, $|\int_{\Omega}<F(\omega),G(\omega)>_{H}d\mu(\omega)|<\infty$, then $F\in L^{p}(\mu,H)$.
\end{lemma}
\begin{proof}
Let $\{\Omega_{n}\}_{n=1}^\infty$ be a family of disjoint measurable subsets of $\Omega$ such that for each $n\geq1$, $\mu(\Omega_{n})<\infty$ and $\Omega=\bigcup_{n=1}^{\infty}\Omega_{n}$. Without loss of generality we can assume $\|F(\omega)\|\neq 0,\,\omega\in\Omega$. Let
$$\Delta_{m}=\{\omega \in \Omega~|~m-1<\|F(\omega)\|\leq m\},\quad m=0,1,2,...~.$$
It is clear that for each $m=0,1,2,...$~, $\Delta_{m}\subseteq\Omega$ is measurable and $\Omega=\bigcup_{m=0,n=1}^{\infty}(\Delta_{m}\cap\Omega_{n})$, where $\{\Delta_{m}\cap\Omega_{n}\}_{n=1,m=1}^{\infty\quad\infty}$ is a family of disjoint and measurable subsets of $\Omega$. We have
$$ \int_{\Omega}\|F(\omega)\|^{p}d\mu(\omega)=
\sum_{m=0}^{\infty}\sum_{n=1}^{\infty}\int_{\Delta_{m}\cap\Omega_{n}}\|F(\omega)\|^{p}d\mu(\omega)$$
and
$$\int_{\Delta_{m}\cap\Omega_{n}}\|F(\omega)\|^{p}d\mu(\omega)\leq m^{p}\mu(\Omega_{n})<\infty. $$
Suppose that $\int_{\Omega}\|F(\omega)\|^{p}d\mu(\omega)=\infty$, then there exists a family $\{E_{k}\}_{k=1}^{\infty}$ of disjoint finite subsets of $\mathbb{N}_{0}\times\mathbb{N}$ such that
$$\sum_{(m,n)\in E_{k}}\int_{\Delta_{m}\cap\Omega_{n}}\|F(\omega)\|^{p}d\mu(\omega)>1.$$
Let $E=\bigcup_{k=1}^{\infty}\bigcup_{(m,n) \in E_{k}}(\Delta_{m}\cap\Omega_{n})$.
Consider $G:\Omega\longrightarrow H$ defined by
$$G(\omega)=\left\{
     \begin{array}{ll}
       c_{k}^{\frac{p}{q}}\|F(\omega)\|^{\frac{p-q}{q}}F(\omega) & if ~\hbox{$\omega\in \bigcup_{(m,n) \in E_{k}}(\Delta_{m}\cap\Omega_{n}),~k=1,2,...$} \\
       0 & if~ \hbox{$\omega\in \Omega\setminus E$}
     \end{array}
   \right.,$$
where
$$c_{k}:=\frac{1}{k^{\frac{q}{p}}}(\int_{\bigcup_{(m,n) \in E_{k}}(\Delta_{m}\cap\Omega_{n})}\|F(\omega)\|^{p}d\mu(\omega))^{-\frac{1}{p}}.$$
Then $G$ is Bochner measurable, and
\begin{align*}
\int_{\Omega}\|G(\omega)\|^{q}d\mu(\omega)=&\int_{E}\|G(\omega)\|^{q}d\mu(\omega)
\\=&\sum_{K=1}^{\infty}\sum_{(m,n)\in E_{k}}\int_{\Delta_{m}\cap\Omega_{n}}\|G(\omega)\|^{q}d\mu(\omega)
\\=&\sum_{K=1}^{\infty}\int_{\bigcup_{(m,n) \in E_{k}}(\Delta_{m}\cap\Omega_{n})}c_{k}^{p}\|F(\omega)\|^{p}d\mu(\omega)
\\=&\sum_{K=1}^{\infty}\frac{1}{k^{q}}<\infty.
\end{align*}
Therefore $G\in L^{q}(\mu,H)$. But
\begin{align*}
|\int_{\Omega}<F(\omega),G(\omega)>_{H}d\mu(\omega)|
=&\sum_{K=1}^{\infty}c_{k}^{\frac{p}{q}}\int_{\bigcup_{(m,n) \in E_{k}}(\Delta_{m}\cap\Omega_{n})}\|F(\omega)\|^{p}d\mu(\omega)
\\=&\sum_{K=1}^{\infty}\frac{1}{k}(\int_{\bigcup_{(m,n) \in E_{k}}(\Delta_{m}\cap\Omega_{n})}\|F(\omega)\|^{p}d\mu(\omega))^{\frac{1}{p}}
\\>&\sum_{K=1}^{\infty}\frac{1}{k}=\infty,
\end{align*}
which is a contradiction.
\end{proof}

The following theorem characterizes Bochner $pg$-Bessel families by operator $T$ defined by (\ref{T}).

\begin{theorem}\label{5}
Suppose that $(\Omega,\Sigma,\mu)$ is a  measure space where $\mu$ is $\sigma$-finite. Let $\{\Lambda_{\omega}\in B(X,H_{\omega}): \omega \in \Omega\}$ be a family such that for each $x\in X$ the mapping $\omega\longmapsto \Lambda_{\omega}(x)$ is Bochner measurable. If  the operator $T$ defined by (\ref{T}) is well-defined and bounded, then $\{\Lambda_{\omega}\}_{\omega \in \Omega}$ is a Bochner $pg$-Bessel family for $X$ with respect to $\{H_{\omega}\}_{\omega\in \Omega}$ with Bessel bound $\|T\|$.
\end{theorem}

\begin{proof}
Let $q$ be the conjugate exponent of $p$ and for $x\in X$, consider
$$F_{x}:L^{q}(\mu,\oplus_{\omega \in \Omega} H_{\omega})\longrightarrow \mathbb{C}$$
$$F_{x}(G)=<x,TG>=\int_{\Omega}<\Lambda_{\omega}(x),G(\omega)>d\mu(\omega),\quad G\in L^{q}(\mu,\oplus_{\omega \in \Omega} H_{\omega}).$$
Then $F_{x}\in (L^{q}(\mu,\oplus_{\omega \in \Omega} H_{\omega}))^{*}$. So $\{\Lambda_{\omega}x\}_{\omega \in \Omega} \in L^{p}(\mu,\oplus_{\omega \in \Omega} H_{\omega})$ by Lemma \ref{4}.
By Remark \ref{00}, $(L^{q}(\mu,\oplus_{\omega \in \Omega} H_{\omega}))^{*}$ and $L^{p}(\mu,\oplus_{\omega \in \Omega} H_{\omega})$ are isometrically isomorphic  and $\|\{\Lambda_{\omega}x\}_{\omega \in \Omega}\|_{p}=\|F_{x}\|$.   Therefore
$$(\int_{\Omega}\|\Lambda_{\omega}x\|^{p}d\mu(\omega))^{\frac{1}{p}}=\|F_{x}\|= \sup_{\|G\|_{q}=1}|<x,TG>|\leq \|T\|\|x\|.$$
\end{proof}

Similar to discrete frames, the analysis operator has closed range.
\begin{lemma}\label{5'}
Let  $\{\Lambda_{\omega}\}_{\omega \in \Omega}$ be a Bochner $pg$-frame for $X$ with respect to $\{H_{\omega}\}_{\omega\in \Omega}$. Then the operator $U$ defined by (\ref{U}) has closed range.
\end{lemma}

\begin{proof}
By assumption, there exist positive constants $A$ and $B$ such that
$$A\|x\|\leq (\int_{\Omega}\|\Lambda_{\omega}(x)\|^{p}d\mu(\omega))^{\frac{1}{p}}\leq B\|x\|,\quad x\in X.$$
By Proposition \ref{T*}, we have
$$A\|x\|\leq\|Ux\|_{p}\leq B\|x\|.$$
Hence $U$ is bounded below. Therefore $U$ has closed range.
\end{proof}

The next proposition shows that there is no Bochner $pg$-frames for a non-reflexive Banach spaces.
\begin{proposition}\label{6}
Let  $\{\Lambda_{\omega}\}_{\omega \in \Omega}$ be a Bochner $pg$-frame for $X$ with respect to $\{H_{\omega}\}_{\omega\in \Omega}$. Then $X$ is reflexive.
\end{proposition}

\begin{proof}
By Lemma \ref{5'}, $R_{U}$ is a closed subspace of $L^{p}(\mu,\oplus_{\omega \in \Omega} H_{\omega})$ and $U:X\longrightarrow R_{U}$ is homeomorphism.  Since
$L^{p}(\mu,\oplus_{\omega \in \Omega} H_{\omega})$ is reflexive, so $X$ is reflexive by Corollary 1.11.22 in \cite{Megginson}.
\end{proof}

In the following lemma we verify the adjoint operators of synthesis and analysis operators.
\begin{lemma}\label{7}
Suppose that $\{\Lambda_{\omega}\}_{\omega \in \Omega}$ is a Bochner $pg$-Bessel family for $X$ with respect to $\{H_{\omega}\}_{\omega\in \Omega}$ with   synthesis operator $T$ and   analysis operator $U$. Then
\\$(i)$ $U^{*}=T$.
\\$(ii)$ If  $\{\Lambda_{\omega}\}_{\omega \in \Omega}$ has the lower Bochner $pg$-frame condition, then $T^{*}J_1=\psi^{*}J_2U$, where $$J_1:X\longrightarrow X^{**}$$ and $$J_2:L^{p}(\mu,\oplus_{\omega \in \Omega} H_{\omega})\longrightarrow (L^{p}(\mu,\oplus_{\omega \in \Omega} H_{\omega}))^{**}$$ are canonical mappings and $\psi$ is the mentioned isometrically isomorphism in Remark \ref {00}.
\end{lemma}

\begin{proof}
$(i)$ For each $G\in L^{q}(\mu,\oplus_{\omega \in \Omega} H_{\omega})$ and $x\in X$, we have
$$<Ux,G>=\int_{\Omega}<\Lambda_{\omega}(x),G(\omega)>d\mu(\omega)=<x,TG>,$$
so $U^*=T$.
\\$(ii)$ Since $X$ and $L^{p}(\mu,\oplus_{\omega \in \Omega} H_{\omega})$ are reflexive, so $J_1$ and $J_2$
are surjective. For each $G\in L^{q}(\mu,\oplus_{\omega \in \Omega} H_{\omega})$ and $x\in X$
\begin{align*}
<G,T^*J_1x>=<TG,J_1x>=<x,TG>=<Ux,G>,
\end{align*}
also
$$<G,\psi^{*}J_2Ux>=<\psi G,J_2Ux>=<Ux,\psi G>=<Ux,G>.$$
Hence $T^{*}J_1=\psi^{*}J_2U$.
\end{proof}

The following theorem characterizes Bochner $pg$-frames by operator $T$  defined by (\ref{T}).

\begin{theorem}\label{8}
Consider the family $\{\Lambda_{\omega}\in B(X,H_{\omega}): \omega \in \Omega\}$.
\\$(i)$ Let $\{\Lambda_{\omega}\}_{\omega \in \Omega}$ be a Bochner $pg$-frame for $X$ with respect to $\{H_{\omega}\}_{\omega\in \Omega}$. Then the operator $T$ defined by (\ref{T}) is a surjective bounded operator.
\\$(ii)$ Let $(\Omega,\Sigma,\mu)$ be a  measure space where $\mu$ is $\sigma$-finite and for each $x\in X$, $\omega \longmapsto \Lambda_{\omega}(x)$ be Bochner measurable. Let  the operator $T$ defined by (\ref{T}) be a surjective bounded operator. Then $\{\Lambda_{\omega}\}_{\omega \in \Omega}$ is a Bochner $pg$-frame for $X$ with respect to $\{H_{\omega}\}_{\omega\in \Omega}$.
\end{theorem}
\begin{proof}
$(i)$ Since $\{\Lambda_{\omega}\}_{\omega \in \Omega}$ is a Bochner $pg$-frame, by Proposition \ref{3}, $T$ is well-defined and bounded.
From the proof of Lemma \ref{5'}, $U$ is bounded below. So, by Lemma \ref{0} and Lemma \ref{7}$(i)$, $U^{*}=T$ is surjective.
\\$(ii)$ Since $T$ is bounded,  $\{\Lambda_{\omega}\}_{\omega \in \Omega}$ is a Bochner $pg$-Bessel family, by Theorem \ref{5}. Since $T=U^{*}$ is surjective, $U$ has a bounded inverse on $R_{U}$ by Lemma \ref{0}. So there exists $A>0$ such that for all $ x\in X$, $\|Ux\|_{p}\geq A\|x\|$. By Proposition \ref{T*}, for all $x\in X$
$$A\|x\|\leq \|Ux\|_{p}=(\int_{\Omega}\|\Lambda_{\omega}(x)\|^{p}d\mu(\omega))^{\frac{1}{p}}.$$
Hence $\{\Lambda_{\omega}\}_{\omega \in \Omega}$ is a  Bochner $pg$-frame.
\end{proof}

\begin{corollary}
If  $\{\Lambda_{\omega}\}_{\omega \in \Omega}$ is  a Bochner $pg$-frame for $X$ with respect to $\{H_{\omega}\}_{\omega\in \Omega}$ and $q$ is the conjugate exponent of $p$
then for each $x^{*}\in X^{*}$ there exists  $G\in L^{q}(\mu,\oplus_{\omega \in \Omega} H_{\omega})$
such that
$$<x,x^{*}>=\int_{\Omega}<\Lambda_{\omega}(x),G(\omega)>d\mu(\omega),\quad  x\in X.$$
\end{corollary}
\begin{proof}
It is obvious.
\end{proof}

The optimal Bochner $pg$-frame bounds can be expressed in terms of synthesis and analysis operators.
\begin{theorem}
Let  $\{\Lambda_{\omega}\}_{\omega \in \Omega}$ be a Bochner $pg$-frame for $X$ with respect to $\{H_{\omega}\}_{\omega\in \Omega}$. Then $\|T\|$ and $\|\tilde{U}\|$ are the optimal upper and lower Bochner $pg$-frame bounds of $\{\Lambda_{\omega}\}_{\omega \in \Omega}$, respectively, where $\tilde{U}$ is the inverse of $U$ on $R_{U}$ and $T$, $U$ are the synthesis and analysis operators of $\{\Lambda_{\omega}\}_{\omega \in \Omega}$, respectively.
\end{theorem}

\begin{proof}
From the proof of Theorem \ref{5}, for each $x\in X$, we have
$$(\int_{\Omega}\|\Lambda_{\omega}x\|^{p}d\mu(\omega))^{\frac{1}{p}}=\|F_{x}\|= \sup_{\|G\|_{q}=1}|<x,TG>|.$$
Therefore
\begin{align*}
&\sup_{\|x\|=1}(\int_{\Omega}\|\Lambda_{\omega}x\|^{p}d\mu(\omega))^{\frac{1}{p}}=\sup_{\|x\|=1}\|F_{x}\|=
\sup_{\|x\|=1}\sup_{\|G\|_{q}=1}|<x,TG>|
\\=&\sup_{\|G\|_{q}=1}\sup_{\|x\|=1}|<x,TG>|
=\sup_{\|G\|_{q}=1}\|TG\|=\|T\|.
\end{align*}
By Proposition \ref{T*}, $\|Ux\|_{p}=(\int_{\Omega}\|\Lambda_{\omega}x\|^{p}d\mu(\omega))^{\frac{1}{p}}$, consequently
$$\inf_{\|x\|=1}\|Ux\|_{p}=\inf_{\|x\|=1}(\int_{\Omega}\|\Lambda_{\omega}x\|^{p}d\mu(\omega))^{\frac{1}{p}}.$$
The operator $U:X\longrightarrow L^{p}(\mu,\oplus_{\omega \in \Omega} H_{\omega})$ is bounded below, so it has bounded inverse $\tilde{U}:R_{U}\longrightarrow X$. We have
$$\inf_{\|x\|=1}\|Ux\|_{p}=\inf_{x\neq0}\frac{\|Ux\|_{p}}{\|x\|}=\inf_{\tilde{Uy}\neq0}\frac{\|y\|_{p}}{\|\tilde{Uy}\|}=
\inf_{y\neq0}\frac{\|y\|_{p}}{\|\tilde{Uy}\|}=\frac{1}{\sup_{y\neq0}\frac{\|\tilde{Uy}\|}{\|y\|_{p}}}=\frac{1}{\|\tilde{U}\|},$$
hence $\inf_{\|x\|=1}(\int_{\Omega}\|\Lambda_{\omega}x\|^{p}d\mu(\omega))^{\frac{1}{p}}=\frac{1}{\|\tilde{U}\|}$.
\end{proof}

\section{Bochner $qg$-Riesz bases}
In this section, we define Bochner $qg$-Riesz bases which are generalization of Riesz bases and characterize their properties.

\begin{definition}\label{Riesz}
Let $1<q<\infty$. A family $\{\Lambda_{\omega}\in B(X,H_{\omega}): \omega \in \Omega\}$ is called a Bochner $qg$-Riesz basis for $X^{*}$ with respect to $\{H_{\omega}\}_{\omega\in \Omega}$, if :
\\$(i)$ $\{x:\Lambda_{\omega}x=0,\,a.e.~[\mu]\}=\{0\}$,
\\$(ii)$ for each $x\in X$, $\omega \longmapsto \Lambda_{\omega}(x)$ is Bochner measurable and the operator
$T$ defined by (\ref{T})
is well-defined and there are positive constants $A$ and $B$ such that
$$ A\|G\|_{q}\leq \|TG\|\leq B\|G\|_{q} ,\quad  G\in L^{q}(\mu,\oplus_{\omega \in \Omega} H_{\omega}).$$
\end{definition}
$A$ and $B$ are called the lower and upper Bochner $qg$-Riesz basis bounds of $\{\Lambda_{\omega}\}_{\omega \in \Omega}$, respectively.

Under some conditions, a Bochner $qg$-Riesz basis is a Bochner $pg$-frame, more precisely:
\begin{proposition}\label{R1}
Suppose that $(\Omega,\Sigma,\mu)$ is a measure space where $\mu$ is $\sigma$-finite and consider the family $\{\Lambda_{\omega}\in B(X,H_{\omega}): \omega \in \Omega\}$.\\
$(i)$ Assume that for each $x\in X$, $\omega \longmapsto \Lambda_{\omega}(x)$ is Bochner measurable. $\{\Lambda_{\omega}\}_{\omega \in \Omega}$ is a Bochner $qg$-Riesz basis for $X^{*}$ with respect to $\{H_{\omega}\}_{\omega\in \Omega}$
if and only if the operator $T$ defined by (\ref{T}) is an invertible bounded operator from $L^{q}(\mu,\oplus_{\omega \in \Omega} H_{\omega})$ onto  $X^{*}$.\\
$(ii)$ Let  $\{\Lambda_{\omega}\}_{\omega \in \Omega}$ be a Bochner $qg$-Riesz basis for $X^{*}$ with respect to $\{H_{\omega}\}_{\omega\in \Omega}$ with the optimal upper Bochner $qg$-Riesz basis bound $B$. If $p$ is the conjugate exponent of $q$ then $\{\Lambda_{\omega}\}_{\omega \in \Omega}$ is a Bochner $pg$-frame for $X$ with respect to $\{H_{\omega}\}_{\omega\in \Omega}$ with  optimal upper Bochner $pg$-frame bound $B$.
\end{proposition}
\begin{proof}$(i)$ By Theorem \ref{5} and Proposition \ref{T*}  and Lemma \ref{7}  and Theorems 3.12, 4.7 and 4.12 in \cite{Rudin}, it is obvious.
\\$(ii)$ By assumption and ($i$), the operator $T$ defined by (\ref{T}) is a bounded invertible operator. So by Theorem \ref{8} ($ii$), $\{\Lambda_{\omega}\}_{\omega \in \Omega}$ is a Bochner $pg$-frame for $X$ with respect to $\{H_{\omega}\}_{\omega\in \Omega}$ with the optimal upper Bochner $pg$-frame bound $B$.
\end{proof}

Next theorem presents some equivalent conditions for a Bochner $pg$-frame being a Bochner $qg$-Riesz basis.
\begin{theorem}\label{11}
Suppose that $(\Omega,\Sigma,\mu)$ is a measure space where $\mu$ is $\sigma$-finite.
Let $\{\Lambda_{\omega}\}_{\omega \in \Omega}$ be a Bochner $pg$-frame for $X$ with respect to $\{H_{\omega}\}_{\omega\in \Omega}$ with  synthesis operator $T$ and  analysis operator $U$ and $q$ be the conjugate exponent of $p$. Then the following statements are equivalent:
\\$(i)$ $\{\Lambda_{\omega}\}_{\omega \in \Omega}$ is a Bochner $qg$-Riesz basis for $X^{*}$.
\\$(ii)$ $T$ is injective.
\\$(iii)$ $R_{U}=L^{p}(\mu,\oplus_{\omega \in \Omega} H_{\omega})$.
\end{theorem}

\begin{proof}
$(i)\rightarrow (ii)$: It is obvious.
\\$(ii)\rightarrow (i)$: By Theorem \ref{8}($i$), the operator $T$ defined by (\ref{T}) is bounded and onto. By $(ii)$, $T$ is also injective. Therefore $T$ has a bounded inverse $T^{-1}:X^{*}\longrightarrow L^{q}(\mu,\oplus_{\omega \in \Omega} H_{\omega})$ and hence $\{\Lambda_{\omega}\}_{\omega \in \Omega}$ is a Bochner $qg$-Riesz basis for $X^{*}$.
\\$(i)\rightarrow (iii)$: By Theorem \ref{R1}, $T$ is invertible, so $T^*$ is invertible.
Lemma \ref{7}($ii$) implies that $R_{U}=L^{p}(\mu,\oplus_{\omega \in \Omega} H_{\omega})$.
\\$(iii)\rightarrow (i)$: Since the operator $U$ is invertible, by Lemma \ref{7} , $T=U^{*}$ is invertible.
\end{proof}

\section{Perturbation of Bochner $pg$-frames}

A perturbation of discrete frame has been discussed in \cite{gj91}. In this section, we present another version of perturbation for Bochner $pg$-frames.


\begin{theorem}
Suppose that $(\Omega,\Sigma,\mu)$ is a measure space where $\mu$ is $\sigma$-finite.
Let $\{\Lambda_{\omega}\}_{\omega \in \Omega}$ be a a Bochner $pg$-frame for $X$ with respect to $\{H_{\omega}\}_{\omega\in \Omega}$ and $q$ be the conjugate exponent of $p$. Let $\{\Theta_{\omega}\in B(H,H_{\omega}): \omega \in \Omega\}$ be a family such that for all $x\in X$, $\omega\longmapsto \Theta_{\omega}(x)$ is Bochner measurable. Assume that there exist constants $\lambda_{1},\lambda_{2},\gamma$ such that
$$0\leq\lambda_2<1, \quad -\lambda_2\leq\lambda_1<1, \quad 0\leq\gamma<(1-\lambda_1-2\lambda_2)A$$ and
\begin{align}\label{per}
&|\int_{\Omega}<(\Lambda_{\omega}-\Theta_{\omega})x,G(\omega)>d\mu(\omega)|\\
\leq&\lambda_{1}|\int_{\Omega}<\Lambda_{\omega}x,G(\omega)>d\mu(\omega)|
+\lambda_{2}|\int_{\Omega}<\Theta_{\omega}x,G(\omega)>d\mu(\omega)|+\gamma\|G\|_{q},\nonumber
\end{align}
for all $G\in L^{q}(\mu,\oplus_{\omega \in \Omega} H_{\omega})$ and $x\in X$. Then $\{\Theta_{\omega}\}_{\omega \in \Omega}$ is a Bochner $pg$-frame  for $X$ with respect to $\{H_{\omega}\}_{\omega\in \Omega}$ with bounds
$$A\big[\frac{(1-\lambda_1-2\lambda_2)-\frac{\gamma}{A}}{1-\lambda_{2}}\big]\qquad and\qquad B\big[\frac{1+\lambda_{1}+\frac{\gamma}{B}}{1-\lambda_{2}}\big],$$
where $A$ and $B$ are  the Bochner $pg$-frame bounds for $\{\Lambda_{\omega}\}_{\omega \in \Omega}$.
\end{theorem}

\begin{proof}
For each $x\in X$ and $G\in L^{q}(\mu,\oplus_{\omega \in \Omega} H_{\omega})$ we have
\begin{align*}
&|\int_{\Omega}<\Theta_{\omega}x,G(\omega)>d\mu(\omega)|
\\\leq&|\int_{\Omega}<(\Lambda_{\omega}-\Theta_{\omega})x,G(\omega)>d\mu(\omega)|
+|\int_{\Omega}<\Lambda_{\omega}x,G(\omega)>d\mu(\omega)|
\\\leq&(1+\lambda_{1})|\int_{\Omega}<\Lambda_{\omega}x,G(\omega)>d\mu(\omega)|
+\lambda_{2}|\int_{\Omega}<\Theta_{\omega}x,G(\omega)>d\mu(\omega)|
\\&+\gamma\|G\|_{q}.
\end{align*}
So
\begin{align*}
&|\int_{\Omega}<\Theta_{\omega}x,G(\omega)>d\mu(\omega)|
\\\leq&\frac{1+\lambda_{1}}{1-\lambda_{2}}|\int_{\Omega}<\Lambda_{\omega}x,G(\omega)>d\mu(\omega)|
+\frac{\gamma}{1-\lambda_{2}}\|G\|_{q}
\\\leq&\frac{1+\lambda_{1}}{1-\lambda_{2}}B\|G\|_{q}\|x\|+\frac{\gamma}{1-\lambda_{2}}\|G\|_{q}
\\=&[\frac{1+\lambda_{1}}{1-\lambda_{2}}B\|x\|+\frac{\gamma}{1-\lambda_{2}}]\|G\|_{q}.
\end{align*}
Now, define $W:L^{q}(\mu,\oplus_{\omega \in \Omega} H_{\omega})\longrightarrow X^*$ by
$$<x,WG>=\int_{\Omega}<\Theta_{\omega}x,G(\omega)>d\mu(\omega),\quad  x\in X,~G\in L^{q}(\mu,\oplus_{\omega \in \Omega} H_{\omega}).$$
Since
\begin{align*}
\|WG\|=\sup_{\|x\|=1}|<x,WG>|=&\sup_{\|x\|=1}|\int_{\Omega}<\Theta_{\omega}x,G(\omega)>d\mu(\omega)|
\\\leq& [\frac{1+\lambda_{1}}{1-\lambda_{2}}B+\frac{\gamma}{1-\lambda_{2}}]\|G\|_{q},
\end{align*}
so $W$ is well-defined and bounded. By Theorem \ref{5}, $\{\Theta_{\omega}\}_{\omega \in \Omega}$ is a
Bochner $pg$-Bessel family  for $X$ with bound $B\big[\frac{1+\lambda_{1}+\frac{\gamma}{B}}{1-\lambda_{2}}\big]$.
\\Now, we show that $\{\Theta_{\omega}\}_{\omega \in \Omega}$ satisfies the lower Bochner $pg$-frame condition. Let $T$ and $U$ be the synthesis and analysis operators of $\{\Lambda_{\omega}\}_{\omega \in \Omega}$, respectively. By Proposition \ref{T*}, for all $x\in X$,
$$A\|x\|\leq\|Ux\|_{p}\leq B\|x\|.$$
By Lemma \ref{5'}, $R_U$ is a closed subspaces of $L^{p}(\mu,\oplus_{\omega \in \Omega} H_{\omega})$, so $Q=U:X\longrightarrow R_U$ is a bijective bounded operator, hence $(Q^{-1})^*:X^*\longrightarrow R_U^*$ is alike. Since
$B^{-1}\leq\|Q^{-1}\|\leq A^{-1}$, so $\|(Q^{-1})^*\|=\|Q^{-1}\|\leq A^{-1}$. Let $x^*\in X^*$ and $S=(Q^{-1})^*$, then $\|S\|\leq A^{-1}$ and $S(x^*)\in R_U^*$, by Hahn-Banach theorem there exists $\varphi\in L^{p}(\mu,\oplus_{\omega \in \Omega} H_{\omega})$ such that $\varphi|_{R_U}=S(x^*)$ and $\|\varphi\|=\|S(x^*)\|$, it follows that
\begin{eqnarray}\label {per2}
\|\varphi\|=\|S(x^*)\|\leq\|S\|\|x^*\|\leq A^{-1}\|x^*\|.
\end{eqnarray}
By Remark \ref{00}, there exists $G\in L^{q}(\mu,\oplus_{\omega \in \Omega} H_{\omega})$ such that $\psi(G)=\varphi$, then
\begin{eqnarray}\label {per3}
\|G\|_{q}=\|\varphi\|\leq A^{-1}\|x^*\|.
\end{eqnarray}
Since $x^*=Q^*(Q^{-1})^*(x^*)$, from (\ref{per}), we have  for each $x\in X$
\begin{align*}
&<x,x^*>=<x,(Q^*S)(x^*)>=<Ux,S(x^*)>=<Ux,\varphi>
\\=&<Ux,\psi(G)>=\int_{\Omega}<\Lambda_{\omega}x,G(\omega)>d\mu(\omega).
\end{align*}
From (\ref{per}) and (\ref{per3}), we obtain that
\begin{align*}
&\|x^*-WG\|
\\=&\sup_{\|x\|=1}|<x,x^*-WG>|
\\=&\sup_{\|x\|=1}|\int_{\Omega}<(\Lambda_{\omega}-\Theta_{\omega})x,G(\omega)>\mu(\omega)|
\\\leq&\sup_{\|x\|=1}\big[\lambda_1|<x,x^*>|+\lambda_2|<x,WG>|+\gamma\|G\|_{q}\big]
\\\leq&\sup_{\|x\|=1}\big[(\lambda_1+\lambda_2)|<x,x^*>|+\lambda_2|<x,WG-x^*>|+\gamma A^{-1}\|x^*\|\big]
\\\leq&\sup_{\|x\|=1}\big[(\lambda_1+\lambda_2)\|x^*\|\|x\|+\lambda_2\|WG-x^*\|\|x\|+\gamma A^{-1}\|x^*\|\big]
\\\leq&(\lambda_1+\lambda_2+\gamma A^{-1})\|x^*\|+\lambda_2\|WG-x^*\|,
\end{align*}
which implies
$$\|WG-x^*\|\leq \frac{(\lambda_1+\lambda_2+)A+\gamma}{(1-\lambda_2)A}\|x^*\|.$$
For a given $x\in X$, there exists $x^*\in X^*$ such that $$\|x^*\|=1,~\|x\|=x^*(x).$$ Hence
\begin{align*}
\|x\|=&x^*(x)=|<x,x^*>|=|<x,x^*-WG>+<x,WG>|
\\\leq&|<x,x^*-WG>|+|<x,WG>|
\\\leq&\frac{(\lambda_1+\lambda_2+)A+\gamma}{(1-\lambda_2)A}\|x^*\|\|x\|
+|\int_{\Omega}<\Theta_{\omega}x,G(\omega)>d\mu(\omega)|
\\\leq&\frac{(\lambda_1+\lambda_2+)A+\gamma}{(1-\lambda_2)A}\|x^*\|\|x\|+
(\int_{\Omega}\|\Theta_{\omega}x\|^{p}d\mu(\omega))^{\frac{1}{p}}\|G\|_{q}
\\\leq&\frac{(\lambda_1+\lambda_2+)A+\gamma}{(1-\lambda_2)A}\|x^*\|\|x\|+A^{-1}\|x^*\|
(\int_{\Omega}\|\Theta_{\omega}x\|^{p}d\mu(\omega))^{\frac{1}{p}},
\end{align*}
therefore
$$\frac{(1-\lambda_1-2\lambda_2)A-\gamma}{1-\lambda_2}\|x\|\leq
(\int_{\Omega}\|\Theta_{\omega}x\|^{p}d\mu(\omega))^{\frac{1}{p}}.$$
\end{proof}


\subsection*{Acknowledgements}
The authors would like to sincerely thank Prof. Dr. Gitta Kutyniok for her valuable comments.

\end{document}